\documentclass[11pt]{article}
\usepackage{amssymb}
\usepackage{eurosym}
\usepackage{amsfonts}
\usepackage{graphicx}
\usepackage{amsmath}

\numberwithin{equation}{section}

\newtheorem{theorem}{Theorem}[section]

\newtheorem{lemma}[theorem]{Lemma}

\newtheorem{proposition}[theorem]{Proposition}

\newenvironment{proof}[1][Proof]{\textbf{#1.} }{\ \rule{0.5em}{0.5em}}
\input{tcilatex}
\begin{document}

\title{On the Finite Dimensional Joint Characteristic Function of L\'{e}%
vy's Stochastic Area Processes}
\author{Xi Geng\thanks{%
Mathematical Institute, University of Oxford, Oxford OX1 3LB, England.
Email: xi.geng@maths.ox.ac.uk} and Zhongmin Qian\thanks{%
Exeter College, Oxford OX1 3DP, England. Email: qianz@maths.ox.ac.uk.}}
\maketitle

\begin{abstract}
The goal of this paper is to derive a formula for the finite dimensional
joint characteristic function (the Fourier transform of the finite
dimensional distribution) of the coupled process $\{(W_{t},L_{t}^{A}):t\in
\lbrack 0,\infty )\}$, where $\{W_{t}:t\in \lbrack 0,\infty )\}$ is a $d$%
-dimensional Brownian motion and $\{L_{t}^{A}:t\in \lbrack 0,\infty )\}$ is
the generalized $d$-dimensional L$\acute{e}$vy's stochastic area process
associated to a $d\times d$ matrix $A.$ Here $A$ need not be skew-symmetric,
and in our computation we allow $A$ to vary. The problem finally reduces to
the solution of a recursive system of symmetric matrix Riccati equations and
a system of independent first order linear matrix ODEs. As an example, the
two dimensional L\'{e}vy's stochastic area process is studied in detail.
\end{abstract}







\section{Introduction}

The L\'{e}vy's stochastic area process $\{L_{t}:t\in \lbrack 0,+\infty )\}$
associated with a two dimensional Brownian motion $%
\{(W_{t}^{(1)},W_{t}^{(2)}):t\in \lbrack 0,+\infty )\}$ was introduced by P.
L\'{e}vy in {\cite{MR0002734}} defined as
\begin{equation*}
L_{t}:=\frac{1}{2}%
\int_{0}^{t}(W_{s}^{(1)}dW_{s}^{(2)}-W_{s}^{(2)}dW_{s}^{(1)})\text{,}\ \ \
t\geqslant 0\text{.}
\end{equation*}%
The geometric meaning of the process $L_{t}$ is the algebraic area enclosed
by the two dimensional Brownian path up to time $t$ and the segment of the
origin $O$ and the point $(W_{t}^{(1)},W_{t}^{(2)})$. By using stochastic
Fourier expansion of Brownian motions, L\'{e}vy derived the conditional
characteristic function of $L_{t}$ with respect to $W_{t}$ as
\begin{equation*}
\mathbb{E}[\exp (i\lambda L_{t})|W_{t}=x]=\frac{t\lambda }{2\sinh (t\lambda
/2)}\exp \left[ \frac{|x|^{2}}{2t}\left( 1-\frac{t\lambda }{2}\coth \frac{%
t\lambda }{2}\right) \right] \text{.}
\end{equation*}

It follows that the characteristic function of $L_{t}$ is given by
\begin{equation*}
\mathbb{E}[\exp (i\lambda L_{t})]=\frac{1}{\cosh (\lambda t/2)}
\end{equation*}%
and the joint characteristic function of the coupled process $%
\{(W_{t},L_{t}):t\geqslant 0\}$ can be computed explicitly.

To compute the finite dimensional joint characteristic function of the
coupled process is harder than the one-dimensional marginal one, since it
involves complicated correlations among different time spots; although the
process is Markovian, it is still nontrivial to derive the transition
density explicitly.

In T. Hida's paper {\cite{MR0301806}}, he considered the problem from a
different and a more general perspective, under the framework of Wiener-It%
\^{o}'s chaos decomposition. He studied the law of quadratic functionals of
Brownian motion on the canonical Wiener space $W$, which are elements in the
second order chaos component. The classical L\'{e}vy's stochastic area
process is a special case in his general setting. The main idea of his
approach is to use the celebrated Wiener-It\^{o}'s chaos decomposition
theorem to construct a one-to-one correspondance between such quadratic
functionals and symmetric Hilbert-Schmidt operators on the Cameron-Martin
subspace $H$ of $W$. By a formal computation, he found out that for any
quadratic functional $F$ on $W$, if $B$ is the corresponding symmetric
Hilbert-Schmidt operator on $H$, then the characteristic function (Fourier
transform) of $F$ is given by
\begin{equation*}
\int_{W}e^{i\lambda F(\omega )}\mu (d\omega )=\prod_{n=1}^{\infty
}(1-2i\lambda \lambda _{n})^{-1/2}e^{-i\lambda \lambda _{n}}\text{,}
\end{equation*}%
where $\{\lambda _{n}:n\geqslant 1\}$ is the family of eigenvalues of $B$
with multiplicity. He also pointed out that the right hand side of the above
identity is equal to
\begin{equation*}
(\det {}_{2}(I-2i\lambda B))^{-1/2}\text{, }
\end{equation*}%
where $\det_{2}$ denotes the regularized Fredholm determinant. By studying
the special example of L\'{e}vy's stochastic area process and computing the
eigenvalues of the corresponding operator, he recovered L\'{e}vy's formula.

After Hida's important work, different methods of computing the regularized
Fredholm determinant in order to study the laws of a wider class of
quadratic Wiener functionals were developed by several Japanese
mathematicians: N. Ikeda, S. Manabe, S. Kusuoka, K. Hara, etc.. The
fundamental ideas in their works can be summerized as two approaches: to
compute the spectrum of the corresponding Hilbert-Schmidt operator directly,
or to reduce the infinite dimensional case to the finite dimensional one by
posing additional assumptions on the operator. The first approach was
developed in Ikeda and Manabe's paper {\cite{MR1354166}}, in which they
computed the spectrum of a variety of quadratic Wiener functionals in order
to study the asymptotic behavior of stochastic oscillatory integrals on the
Wiener space. The second approach was developed from different aspects based
on a fundamental decomposition assumption proposed by Ikeda, Kusuoka and
Manabe. They restricted themselves to the study of operators which could be
divided into an operator of finite dimensional range and an operator of
Volterra's type, in order to reduce the case to a finite dimensional one in
a certain sense. In Ikeda, Kusuoka and Manabe's paper \cite{MR1335477},
based on such decomposition, the computation of the regularized determinant
is reduced to the computation of a finite dimensional determinant, in which
algebraic methods and differential equations methods can be applied. Later
in {\cite{MR1266245}} they developed a general method to compute the law of
the corresponding type of quadratic Wiener functionls by using the
Cameron-Martin transformation along critical paths. In the meanwhile, they
put forward another idea for the computation by relating the problem to the Van
Vleck-Pauli formula in quantum mechanics and derived a formula for the law
of a certain class of quadratic Wiener functionals by using physical
approach. In the paper {\cite{MR1869989}} of Hara and Ikeda, they further
developed the ideas in {\cite{MR1266245}} by relating the problem to the
study of dynamics on Grassmannians. Due to the fact that the determinant is
expressed by the solutions of the Jacobi equation with Van Vleck-Pauli's
formula, they showed that the determinant could be computed in terms of Pl%
\"{u}chker coordinates of Grassmannian manifolds.

From the works on computing the regularized determinant, one can see that
the computation is very involved, even reduced to the finite dimensional
case. In the series of works mentioned based on Hida's formula, they
considered the marginal law of the two dimensional L\'{e}vy's stochastic
area process and recovered L\'{e}vy's formula.

It should be pointed out that after L\'{e}vy's original work, several
mathematicians proposed different methods to compute the marginal
characteristic function of the two dimensional L\'{e}vy's stochastic area process
directly. In particular, B. Gaveau {\cite{MR0461589}} studied the marginal
distribution of the process (inversion of the characteristic function). In
the present paper, we are going to study the finite dimensional joint
characteristic funtion of the $d$-dimensional generalized L\'{e}vy's
stochastic area processes from a different angle of view by computing the
multi-dimensional Fourier transform directly. Our work is based on the idea
of K. Helmes and A. Schwane in their paper {\cite{MR724703}} for the
computation of the marginal characteristic function of the $d$-dimensional
Brownian motion together with the general $d$-dimensional L\'{e}vy's
stochastic area processes.

Let $\{W_{t}:t\in \lbrack 0,\infty )\}$ be a $d$-dimensional Brownian
motion, and let $so(d)$ be the space of $d\times d$ skew-symmetric matrices,
where $d\geqslant 2$. Fix $A\in so(d)$, the process
\begin{equation*}
L_{t}^{A}:=\int_{0}^{t}\langle AW_{s},dW_{s}\rangle ,\ \ \ t\in \lbrack
0,\infty ),
\end{equation*}%
is called the general $d$-dimensional L\'{e}vy's stochastic area process
associated to $A$. In Helmes and Schwane's paper {\cite{MR724703}}, they
derived a formula for the marginal characteristic funcion of the coupled
process $\{(W_{t},L_{t}^{A}):t\in \lbrack 0,\infty )\}$, based on the
Cameron-Martin-Girsanov's transformation theorem and It\^{o}'s formula. In
our paper, we are going to further exploit this method to establish the
finite dimensional joint characteristic function (the Fourier transform of
the finite dimensional distribution) of the coupled process in a more
general setting. We will see that it reduces to the solution of a recursive
system of symmetric matrix Riccati equations and a system of independent
first order linear matrix ODEs.

Since the computation is quite lengthy, we first present the main result in
our paper and the main idea for the proof. All details and technical steps
are left to section 2. As an example, we study the two dimensional L\'{e}%
vy's stochastic area process in detail.

Assume that $\{W_{t},\mathcal{F}_{t}:t\geqslant 0\}$ is a $d$-dimensional
Brownian motion on some probability space $(\Omega ,\mathcal{F},\mathbb{P})$%
. For any $A$ being a real $d\times d$ matrix (not necessarily
skew-symmetric), we define the generalized L\'{e}vy's stochastic area
process $\{L_{t}^{A}:t\geqslant 0\}$ associated to $A$ as
\begin{equation*}
L_{t}^{A}:=\int_{0}^{t}\langle AW_{s},dW_{s}\rangle \text{,}\ \ \ t\geqslant
0\text{.}
\end{equation*}%
Fix $0=t_{0}<t_{1}<t_{2}<\cdots <t_{n}<\infty $,$\ n\geqslant 1$. Assume
that $A_{1},\cdots ,A_{n}$ are real $d\times d$ matrices. For $\gamma
_{1},\cdots ,\gamma _{n}\in \mathbb{R}^{d}$, $\Lambda _{1},\cdots ,\Lambda
_{n}\in \mathbb{R}$, we are going to compute the following functional:
\begin{equation*}
f(\gamma _{1},\cdots ,\gamma _{n};\Lambda _{1},\cdots ,\Lambda _{n}):=%
\mathbb{E}\left[ \exp \left( \sum_{k=1}^{n}i\langle \gamma
_{k},W_{t_{k}}\rangle +\sum_{k=1}^{n}i\Lambda _{k}L_{t_{k}}^{A_{k}}\right) %
\right] \text{.}
\end{equation*}%
where $i=\sqrt{-1}$. Notice that the matrix can vary on different time spots.

Throughout this paper, $0<t_{1}<\cdots <t_{n}<\infty $ and $A_{1},\cdots
,A_{n}\in M^{d}(\mathbb{R})$ will be fixed. The notation $\ast $ will denote
the transpose operator. We should point out that even for complex matrices, $%
\ast $ is simply the transpose without taking conjugate. For $%
z=(z^{1},\cdots ,z^{j})\in \mathbb{C}^{n}$, we use $\langle z\rangle {}^{2}$
to denote $\sum_{j=1}^{n}(z^{j})^{2}$ which is the holomorphic extension of
the real case. $Tr$ will be denoted as the trace operator.

Our main result is the following.

\begin{theorem}
\label{mth1}The functional $f(\gamma _{1},\cdots ,\gamma _{n};\Lambda
_{1},\cdots ,\Lambda _{n})$ is determined by
\begin{eqnarray*}
f(\gamma _{1},\cdots ,\gamma _{n};\Lambda _{1},\cdots ,\Lambda _{n})
&=&\prod_{j=1}^{n}\exp \{\frac{1}{2}\int_{0}^{t_{j}}Tr(K_{i\Lambda
_{j},\cdots ,i\Lambda _{n}}(s))ds \\
&&-\frac{1}{2}\int_{t_{j-1}}^{t_{j}}\langle H_{i\Lambda _{j},\cdots
,i\Lambda _{n}}^{\ast -1}(s)H_{i\Lambda _{j},\cdots ,i\Lambda _{n}}^{\ast
}(t_{j})\mu _{j}\rangle ^{2}ds\}.
\end{eqnarray*}%
Here $\{K_{i\Lambda _{j},\cdots ,i\Lambda _{n}}(t):t\in \lbrack
0,t_{j}],j=1,\cdots ,n\}$ is defined as the solution of the recursive system
of $n$ symmetric matrix Riccati equations (starting from $j=n$ to $j=1$):
\begin{eqnarray*}
\frac{d}{dt}K_{i\Lambda _{j},\cdots ,i\Lambda _{n}}(t) &=&C_{i\Lambda
_{j},\cdots ,i\Lambda _{n}}(t)-K_{i\Lambda _{j},\cdots ,i\Lambda
_{n}}(t)(\sum_{r=j+1}^{n}K_{i\Lambda _{r},\cdots ,i\Lambda _{n}}(t) \\
&&+\sum_{r=j}^{n}(i\Lambda _{r}A_{r}))-(\sum_{r=j+1}^{n}K_{i\Lambda
_{r},\cdots ,i\Lambda _{n}}(t) \\
&&+\sum_{r=j}^{n}(i\Lambda _{r}A_{r}^{\ast }))K_{i\Lambda _{j},\cdots
,i\Lambda _{n}}(t)-K_{i\Lambda _{j},\cdots ,i\Lambda _{n}}^{2}(t), \\
&&t\in \lbrack 0,t_{j}], \\
K_{i\Lambda _{j},\cdots ,i\Lambda _{n}}(t_{j}) &=&0,
\end{eqnarray*}%
where
\begin{eqnarray*}
C_{i\Lambda _{j},\cdots ,i\Lambda _{n}}(t):= &&\Lambda _{j}^{2}A_{j}^{\ast
}A_{j}-i\Lambda _{j}[(\sum_{r=j+1}^{n}K_{i\Lambda _{r},\cdots ,i\Lambda
_{n}}(t)+\sum_{r=j+1}^{n}(i\Lambda _{r}A_{r}^{\ast }))A_{j} \\
&&+A_{j}^{\ast }(\sum_{r=j+1}^{n}K_{i\Lambda _{r},\cdots ,i\Lambda
_{n}}(t)+\sum_{r=j+1}^{n}(i\Lambda _{r}A_{r}))],\ \ \ t\in \lbrack 0,t_{j}].
\end{eqnarray*}%
Moreover, $\{H_{i\Lambda _{j},\cdots ,i\Lambda _{n}}(t):t\in \lbrack
0,t_{j}],j=1,\cdots ,n\}$ is the solution of the system of $n$ independent
linear matrix ODEs:
\begin{eqnarray*}
\frac{d}{dt}H_{i\Lambda _{j},\cdots ,i\Lambda _{n}}(t)
&=&(\sum_{r=j}^{n}K_{i\Lambda _{r},\cdots ,i\Lambda
_{n}}(t)+\sum_{r=j}^{n}(i\Lambda _{r}A_{r}))H_{i\Lambda _{j},\cdots
,i\Lambda _{n}}(t), \\
&&t\in \lbrack 0,t_{j}], \\
H_{i\Lambda _{j},\cdots ,i\Lambda _{n}}(0) &=&Id, \\
&&j=1,\cdots ,n,
\end{eqnarray*}%
and $\{\mu _{j}:j=1,\cdots ,n\}$ is defined recursively as
\begin{eqnarray*}
\mu _{n}:= &&\gamma _{n}, \\
\mu _{j}:= &&\gamma _{j}+H_{i\Lambda _{j+1},\cdots ,i\Lambda _{n}}^{\ast
-1}(t_{j})H_{i\Lambda _{j+1},\cdots ,i\Lambda _{n}}^{\ast }(t_{j+1})\mu
_{j+1}, \\
&&j=n-1,\cdots ,1.
\end{eqnarray*}
\end{theorem}

The subscripts of $K$ and $H$ are used to indicate the dependence on $%
\Lambda $. We do not wish to address the existence and uniqueness of the
preceeding maxtrix differential system of Riccati type. Under our context,
there is a unique solution to the system appearing in Theorem \ref{mth1},
for details, the reader may consult \cite{MR0108628}, \cite{MR0015610}, \cite%
{MR0357936} and etc.

Now we are going to briefly express the main idea of the proof in an
informal way.

The fundamental tool of the proof is Girsanov's transformation. However, it
fits for real-valued processes only. Hence we focus on the real case first,
that is, consider the functinal
\begin{equation*}
g(\gamma_{1},\cdots,\gamma_{n};\lambda_{1},\cdots,\lambda_{n}):=\mathbb{E}%
[\exp\{\sum_{k=1}^{n}i\langle\gamma_{k},W_{t_{k}}\rangle+\sum_{k=1}^{n}%
\lambda_{k}L_{t_{k}}^{A_{k}}\}]
\end{equation*}
first for small $\lambda_{1},\cdots,\lambda_{n}\in\mathbb{R},$ and then try
to complexify the result by standard arguments in complex analysis.

Our idea of computing the functional $g$ is to eliminate the stochastic
integrals one by one, starting from the largest time interval $[0,t_{n}],$
by using the Girsanov's transformation theorem, and try to track the
original Brownian motion along each time of transformation in order to
handle the $W$$.$ When applying change of measure, a stochastic integral
(with respect to a proper Brownian motion) is transformed to one-half of its
quadratic variation, which is a Lebesgue integral almost surely. To handle
this term, we introduce a symmetric matrix Riccati equation to split it into
three terms, which will be transformed to a deterministic one by change of
measure again and applying It\^{o}'s formula for a suitable process.
Recursively, we will be able to cancel out all of the $n$ stochastic
integrals in $g$, and obtain a recursive system of symmetric matrix Riccati
equations. To handle the $W$ term, we need to track the change of $W$ along
every time of transformation. From Girsanov's theorem, we will see that the
diffusion form of $W$ under each transformation is actually invariant (it is
always Gaussian by solving a linear SDE), which will enable us to do the
computation easily. This procedure will lead us to a system of independent
first order linear matrix ODEs.

Our interest in looking for an (as far as possible) explicit formula for the
joint law of Brownian motion together with its L\'{e}vy area is motivated by
the recent progress in the understanding of the solutions of It\^{o}'s
stochastic differential equations revealed recently in T. Lyons' work {\cite%
{MR1654527}}, in which it has been demonstrated that a large class of Wiener
functionals (including It\^{o} solutions to SDEs) are continuous with
respect to Brownian sample paths \emph{and} their L\'{e}vy areas equipped
with the law of Brownian motion together with the L\'{e}vy area process. For
details about these developments, see T. Lyons {\cite{MR1654527}}, T. Lyons
and Z. Qian {\cite{MR2036784}}, P.K. Friz and N.B. Victoir \cite{MR2604669}.

\section{Proof of the Main Result}

Now we are going to work out the idea in section 1. The procedure is to
handle the real case first and then use complexification.

The following technical lemma is important (see also \cite{MR724703}), which
gives us finiteness in the real case.

\begin{lemma}
There exists a number $c>0$, such that
\begin{equation*}
\sup\{\mathbb{E}[\exp\{\sum_{k=1}^{n}\lambda_{k}L_{t_{k}}^{A_{k}}\}]:%
\lambda_{j}\in(-c,c),j=1,2,\cdots,n\}<\infty.
\end{equation*}
\end{lemma}

\begin{proof}
It suffices to consider the case when $n=1$. By Cauchy-Schwarz's inequality,
we have
\begin{eqnarray*}
\mathbb{E}[\exp\{\lambda L_{t}^{A}\}] & = & \mathbb{E}[\exp\{\lambda%
\int_{0}^{t}\langle AW_{s},dW_{s}\rangle\}] \\
& = & \mathbb{E}[\exp\{\lambda\int_{0}^{t}\langle
AW_{s},dW_{s}\rangle-\lambda^{2}\int_{0}^{t}|AW_{s}|^{2}ds\}+\lambda^{2}%
\int_{0}^{t}|AW_{s}|^{2}ds\}] \\
& \leqslant & \mathbb{E}^{1/2}[\exp\{2\lambda\int_{0}^{t}\langle
AW_{s},dW_{s}\rangle-\frac{(2\lambda)^{2}}{2}\int_{0}^{t}|AW_{s}|^{2}ds\}]%
\cdot \\
& & \mathbb{E}^{1/2}[\exp\{2\lambda^{2}\int_{0}^{t}|AW_{s}|^{2}ds\}].
\end{eqnarray*}
Notice that the first term inside the expectation $\mathbb{E}$ on the right
hand side is a local martingle, the finiteness of $\mathbb{E}[\exp\{\lambda
L_{t}^{A}\}]$ will follow immediately once we show that
\begin{equation*}
\mathbb{E}[\exp\{\lambda^{2}\int_{0}^{t}|AW_{s}|^{2}ds\}]<\infty
\end{equation*}
when $\lambda$ is small enough. By Jensen's inequality, it remains to show
that
\begin{equation*}
\sup_{s\in[0,t]}\mathbb{E}[\exp\{\lambda^{2}|AW_{s}|^{2}\}]<\infty
\end{equation*}
when $\lambda$ is small enough, which is obvious by simple calculation based
on Gaussian random variables.
\end{proof}

Now consider small $\lambda_{1},\cdots,\lambda_{n}\in\mathbb{R}$ as in lemma
1, $\gamma_{1},\cdots,\gamma_{n}\in\mathbb{R},$ and the functional
\begin{equation*}
g(\gamma_{1},\cdots,\gamma_{n};\lambda_{1},\cdots,\lambda_{n})=\mathbb{E}%
[\exp\{\sum_{k=1}^{n}i\langle\gamma_{k},W_{t_{k}}\rangle+\sum_{k=1}^{n}%
\lambda_{k}L_{t_{k}}^{A_{k}}\}]
\end{equation*}
defined in section one. The following proposition gives the formula of $g.$

\begin{proposition}
The function $g(\gamma_{1},\cdots,\gamma_{n};\lambda_{1},\cdots,\lambda_{n})$
is determined by
\begin{align*}
& g(\gamma_{1},\cdots,\gamma_{n};\lambda_{1},\cdots,\lambda_{n}) \\
= & \prod_{j=1}^{n}\exp\{\frac{1}{2}\int_{0}^{t_{j}}Tr(K_{\lambda_{j},%
\cdots,\lambda_{n}}(s))ds-\frac{1}{2}\int_{t_{j-1}}^{t_{j}}|H_{\lambda_{j},%
\cdots,\lambda_{n}}^{*-1}(s)H_{\lambda_{j},\cdots,\lambda_{n}}^{*}(t_{j})%
\mu_{j}|^{2}ds\}.
\end{align*}
Here $\{K_{\lambda_{j},\cdots,\lambda_{n}}(t):t\in[0,t_{j}],j=1,2,\cdots,n\}$
is defined recursively (starting from $j=n$) by the symmetric matrix Riccati
equation
\begin{eqnarray*}
\frac{d}{dt}K_{\lambda_{j},\cdots,\lambda_{n}}(t) & = &
C_{\lambda_{j},\cdots,\lambda_{n}}(t)-K_{\lambda_{j},\cdots,\lambda_{n}}(t)(%
\sum_{r=j+1}^{n}K_{\lambda_{r},\cdots,\lambda_{n}}(t)+\sum_{r=j}^{n}(%
\lambda_{r}A_{r})) \\
& &
-(\sum_{r=j+1}^{n}K_{\lambda_{r},\cdots,\lambda_{n}}(t)+\sum_{r=j}^{n}(%
\lambda_{r}A_{r}^{*}))K_{\lambda_{j},\cdots,\lambda_{n}}(t)-K_{\lambda_{j},%
\cdots,\lambda_{n}}^{2}(t), \\
& & t\in[0,t_{j}], \\
K_{\lambda_{j},\cdots,\lambda_{n}}(t_{j}) & = & 0,
\end{eqnarray*}
where
\begin{eqnarray*}
C_{\lambda_{j},\cdots,\lambda_{n}}(t): & = &
-\lambda_{j}^{2}A_{j}^{*}A_{j}-\lambda_{j}[(\sum_{r=j+1}^{n}K_{\lambda_{r},%
\cdots,\lambda_{n}}(t)+\sum_{r=j+1}^{n}(\lambda_{r}A_{r}^{*}))A_{j} \\
& &
+A_{j}^{*}(\sum_{r=j+1}^{n}K_{\lambda_{r},\cdots,\lambda_{n}}(t)+%
\sum_{r=j+1}^{n}(\lambda_{r}A_{r}))],\ \ \ t\in[0,t_{j}]
\end{eqnarray*}
is also defined recursively starting from $j=n$. Moreover, $%
\{H_{\lambda_{j},\cdots,\lambda_{n}}(t):t\in[0,t_{j}],j=1,2,\cdots,n\}$ is
the solution of the system of $n$ independent first order linear matrix ODEs
\begin{eqnarray*}
\frac{d}{dt}H_{\lambda_{j},\cdots,\lambda_{n}}(t) & = &
(\sum_{r=j}^{n}K_{\lambda_{r},\cdots,\lambda_{n}}(t)+\sum_{r=j}^{n}(%
\lambda_{j}A_{j}))H_{\lambda_{j},\cdots,\lambda_{n}},\ \ \ t\in[0,t_{j}], \\
H_{\lambda_{j},\cdots,\lambda_{n}}(0) & = & Id, \\
& & j=1,2,\cdots,n,
\end{eqnarray*}
and $\{\mu_{j}:j=1,2,\cdots,n\}$ is defined recursively by
\begin{eqnarray*}
\mu_{n} & := & \gamma_{n}, \\
\mu_{j} & := &
\gamma_{j}+H_{\lambda_{j+1},\cdots,\lambda_{n}}^{*-1}(t_{j})H_{%
\lambda_{j+1},\cdots,\lambda_{n}}^{*}(t_{j+1})\mu_{j+1},\ \ \
j=1,2,\cdots,n-1.
\end{eqnarray*}
\end{proposition}

\begin{proof}
We divide our proof into two steps.

(1) Step one.

Consider first that $\gamma _{1}=\cdots \gamma _{n}=0,$ and write
\begin{eqnarray*}
h(\lambda _{1},\cdots ,\lambda _{n}):= &&\mathbb{E}[\exp
\{\sum_{k=1}^{n}\lambda _{k}L_{t_{k}}^{A_{k}}\}]. \\
&=&\mathbb{E}[\exp \{\sum_{k=1}^{n}\lambda _{k}\int_{0}^{t_{k}}\langle
A_{k}W_{s},dW_{s}\rangle \}].
\end{eqnarray*}%
By changing the original probability measure $P$ on the largest time
interval $[0,t_{n}]$, we have
\begin{eqnarray*}
h(\lambda _{1},\cdots ,\lambda _{n}) &=&\mathbb{E}[\exp
\{\sum_{k=1}^{n-1}\lambda _{k}\int_{0}^{t_{k}}\langle
A_{k}W_{s},dW_{s}\rangle +\lambda _{n}\int_{0}^{t_{n}}\langle
A_{n}W_{s},dW_{s}\rangle \\
&&-\frac{\lambda _{n}^{2}}{2}\int_{0}^{t_{n}}|A_{n}W_{s}|^{2}ds+\frac{%
\lambda _{n}^{2}}{2}\int_{0}^{t_{n}}|A_{n}W_{s}|^{2}ds\}] \\
&=&\mathbb{E}_{n}[\exp \{\sum_{k=1}^{n-1}\lambda _{k}\int_{0}^{t_{k}}\langle
A_{k}W_{s},dW_{s}\rangle +\frac{\lambda _{n}^{2}}{2}%
\int_{0}^{t_{n}}|A_{n}W_{s}|^{2}ds\}],
\end{eqnarray*}%
where $\mathbb{E}_{n}$ denotes the expectation under the probability measure
\begin{equation*}
dP_{n}:=\exp \{\lambda _{n}\int_{0}^{t_{n}}\langle A_{n}W_{s},dW_{s}\rangle -%
\frac{\lambda _{n}^{2}}{2}\int_{0}^{t_{n}}|A_{n}W_{s}|^{2}ds\}dP.
\end{equation*}%
Notice that $\{W_{t}:t\in \lbrack 0,t_{n}]\}$ may not be a Brownian motion
under the new measure $P_{n}.$ However, by the Girsanov's theorem, under $%
P_{n},$ the process
\begin{equation*}
W_{t}^{(n)}:=W_{t}-\lambda _{n}\int_{0}^{t}A_{n}W_{s}ds,\ \ \ t\in \lbrack
0,t_{n}],
\end{equation*}%
is a Brownian motion, and the original process $\{W_{t}:t\in \lbrack
0,t_{n}]\}$ satisfies the following SDE:
\begin{equation*}
dW_{t}=\lambda _{n}A_{n}W_{t}dt+dW_{t}^{(n)},\ \ \ t\in \lbrack 0,t_{n}].
\end{equation*}%
In order to eliminate the integral over $[0,t_{n}],$ let $C_{n}(t):=-\lambda
_{n}^{2}A_{n}^{\ast }A_{n}\ (t\in \lbrack 0,t_{n}])$ and introduce the
following matrix Riccati equation:
\begin{eqnarray*}
\frac{d}{dt}K_{\lambda _{n}}(t) &=&C_{n}(t)-\lambda _{n}[K_{\lambda
_{n}}(t)A_{n}+A_{n}^{\ast }K_{\lambda _{n}}(t)]-K_{\lambda _{n}}^{2}(t),\ \
\ t\in \lbrack 0,t_{n}], \\
K_{\lambda _{n}}(t_{n}) &=&0.
\end{eqnarray*}%
From now on, to simplify our notation, we will use $K_{n}$ to denote $%
K_{\lambda _{n}}$, and later we will also use $K_{j}$ to denote $K_{\lambda
_{j},\cdots ,\lambda _{n}},\ j=1,\cdots ,n-1$ to omit the indication on the
dependence on $\lambda .$ By symmetricity of the equation, the unique
solution $\{K_{n}(t):t\in \lbrack 0,t_{n}]\}$ is symmetric. Hence, by
substitution, we have
\begin{eqnarray*}
h(\lambda _{1},\cdots ,\lambda _{n}) &=&\mathbb{E}_{n}[\exp
\{\sum_{k=1}^{n-1}\lambda _{k}\int_{0}^{t_{k}}\langle
A_{k}W_{s},dW_{s}\rangle -\frac{1}{2}\int_{0}^{t_{n}}W_{s}^{\ast
}C_{n}(s)W_{s}ds\}] \\
&=&\mathbb{E}_{n}[\exp \{\sum_{k=1}^{n-1}\lambda _{k}\int_{0}^{t_{k}}\langle
A_{k}W_{s},dW_{s}\rangle -\frac{1}{2}\int_{0}^{t_{n}}W_{s}^{\ast }\frac{d}{ds%
}K_{n}(s)W_{s}ds \\
&&-\frac{1}{2}\int_{0}^{t_{n}}W_{s}^{\ast }[K_{n}(s)(\lambda
_{n}A_{n})+\lambda _{n}A_{n}^{\ast }K_{n}(s)]W_{s}ds \\
&&-\frac{1}{2}\int_{0}^{t_{n}}|K_{n}(s)W_{s}|^{2}ds\}] \\
&=&\mathbb{E}_{n}[\exp \{\sum_{k=1}^{n-1}\lambda _{k}\int_{0}^{t_{k}}\langle
A_{k}W_{s},dW_{s}\rangle -\frac{1}{2}\int_{0}^{t_{n}}W_{s}^{\ast }\frac{d}{ds%
}K_{n}(s)W_{s}ds \\
&&-\int_{0}^{t_{n}}\langle K_{n}(s)W_{s},\lambda _{n}A_{n}W_{s}\rangle ds-%
\frac{1}{2}\int_{0}^{t_{n}}|K_{n}(s)W_{s}|^{2}ds\}].
\end{eqnarray*}%
By changing of measure again,
\begin{eqnarray*}
h(\lambda _{1},\cdots ,\lambda _{n}) &=&\widetilde{\mathbb{E}_{n}}[\exp
\{\sum_{k=1}^{n-1}\lambda _{k}\int_{0}^{t_{k}}\langle
A_{k}W_{s},dW_{s}\rangle -\frac{1}{2}\int_{0}^{t_{n}}W_{s}^{\ast }\frac{d}{ds%
}K_{n}(s)W_{s}ds \\
&&-\int_{0}^{t_{n}}\langle K_{n}(s)W_{s},\lambda _{n}A_{n}W_{s}\rangle
ds-\int_{0}^{t_{n}}\langle K_{n}(s)W_{s},dW_{s}^{(n)}\rangle \}] \\
&=&\widetilde{\mathbb{E}_{n}}[\exp \{\sum_{k=1}^{n-1}\lambda
_{k}\int_{0}^{t_{k}}\langle A_{k}W_{s},dW_{s}\rangle -\frac{1}{2}%
\int_{0}^{t_{n}}W_{s}^{\ast }\frac{d}{ds}K_{n}(s)W_{s}ds \\
&&-\int_{0}^{t_{n}}\langle K_{n}(s)W_{s},\lambda _{n}A_{n}W_{s}\rangle ds \\
&&-\int_{0}^{t_{n}}\langle K_{n}(s)W_{s},dW_{s}-\lambda
_{n}A_{n}W_{s}ds\rangle \}] \\
&=&\widetilde{\mathbb{E}_{n}}[\exp \{\sum_{k=1}^{n-1}\lambda
_{k}\int_{0}^{t_{k}}\langle A_{k}W_{s},dW_{s}\rangle -\frac{1}{2}%
\int_{0}^{t_{n}}W_{s}^{\ast }\frac{d}{ds}K_{n}(s)W_{s}ds \\
&&-\int_{0}^{t_{n}}\langle K_{n}(s)W_{s},dW_{s}\rangle \}].
\end{eqnarray*}%
where $\widetilde{\mathbb{E}_{n}}$ denotes the expectation under the
probability measure
\begin{equation*}
d\widetilde{P_{n}}:=\exp \{\int_{0}^{t_{n}}\langle
K_{n}(s)W_{s},dW_{s}^{(n)}\rangle -\frac{1}{2}%
\int_{0}^{t_{n}}|K_{n}(s)W_{s}|^{2}ds\}dP_{n}.
\end{equation*}%
Under $\widetilde{P_{n}}$, the process
\begin{equation*}
\widetilde{W_{t}^{(n)}}:=W_{t}^{(n)}-\int_{0}^{t}K_{n}(s)W_{s}ds,\ \ \ t\in
\lbrack 0,t_{n}]
\end{equation*}%
is a Brownian motion, and the original process $\{W_{t}:t\in \lbrack
0,t_{n}]\}$ satisfies the following SDE:
\begin{equation*}
dW_{t}=(K_{n}(t)+\lambda _{n}A_{n})W_{t}dt+d\widetilde{W_{t}^{(n)}},\ \ \
t\in \lbrack 0,t_{n}].
\end{equation*}%
It should be pointed out that under $\widetilde{P_{n}}$, the quadratic
variation process of the semi-martingale $\{W_{t}:t\in \lbrack 0,t_{n}]\}$
is actually the same as that of the Browinian motion $\{\widetilde{%
W_{t}^{(n)}}:t\in \lbrack 0,t_{n}]\}.$ Now let $F(t,w):\ [0,t_{n}]\times
\mathbb{R}^{d}\rightarrow \mathbb{R}$ be defined as
\begin{equation*}
F(t,w):=w^{\ast }K_{n}(t)w,
\end{equation*}%
by applying It\^{o}'s formula to the process $\{F(t,W_{t}):t\in \lbrack
0,t_{n}]\},$ we have
\begin{equation*}
\int_{0}^{t_{n}}W_{s}^{\ast }\frac{d}{ds}K_{n}(s)W_{s}ds+2\int_{0}^{t_{n}}%
\langle K_{n}(s)W_{s},dW_{s}\rangle +\int_{0}^{t_{n}}Tr(K(s))ds=0,
\end{equation*}%
where $Tr$ denotes the trace operator. Therefore, we arrive at
\begin{equation*}
h(\lambda _{1},\cdots ,\lambda _{n})=\exp \{\frac{1}{2}%
\int_{0}^{t_{n}}Tr(K_{n}(s))ds\}\cdot \widetilde{\mathbb{E}_{n}}[\exp
\{\sum_{k=1}^{n-1}\lambda _{k}\int_{0}^{t_{k}}\langle
A_{k}W_{s},dW_{s}\rangle \}].
\end{equation*}

Now we proceed a similar argument over the second largest time interval $%
[0,t_{n-1}].$ The difference here is that $\{W_{t}:t\in \lbrack 0,t_{n-1}]\}$
is not a Brownian motion under the probability measure $\widetilde{P_{n}}.$
However, still by changing of measure, we have
\begin{align*}
& h(\lambda _{1},\cdots ,\lambda _{n})\exp \{-\frac{1}{2}%
\int_{0}^{t_{n}}Tr(K_{n}(s))ds\} \\
=& \widetilde{E_{n}}[\exp \{\sum_{k=1}^{n-2}\lambda
_{k}\int_{0}^{t_{k}}\langle A_{k}W_{s},dW_{s}\rangle +\lambda
_{n-1}\int_{0}^{t_{n-1}}\langle A_{n-1}W_{s},dW_{s}\rangle \}] \\
=& \widetilde{E_{n}}[\exp \{\sum_{k=1}^{n-2}\lambda
_{k}\int_{0}^{t_{k}}\langle A_{k}W_{s},dW_{s}\rangle +\lambda
_{n-1}\int_{0}^{t_{n-1}}\langle A_{n-1}W_{s},(K_{n}(s)+\lambda
_{n}A_{n})W_{s}\rangle ds \\
& +\lambda _{n-1}\int_{0}^{t_{n-1}}\langle A_{n-1}W_{s},d\widetilde{%
W_{s}^{(n)}}\rangle -\frac{\lambda _{n-1}^{2}}{2}%
\int_{0}^{t_{n-1}}|A_{n-1}W_{s}|^{2}ds \\
& +\frac{\lambda _{n-1}^{2}}{2}\int_{0}^{t_{n-1}}|A_{n-1}W_{s}|^{2}ds\}] \\
=& \mathbb{E}_{n-1}[\exp \{\sum_{k=1}^{n-2}\lambda
_{k}\int_{0}^{t_{k}}\langle A_{k}W_{s},dW_{s}\rangle +\lambda
_{n-1}\int_{0}^{t_{n-1}}\langle A_{n-1}W_{s},(K_{n}(s)+\lambda
_{n}A_{n})W_{s}\rangle ds \\
& +\frac{\lambda _{n-1}^{2}}{2}\int_{0}^{t_{n-1}}|A_{n-1}W_{s}|^{2}ds\}].
\end{align*}%
Here $\mathbb{E}_{n-1}$ is the expectation under the probability measure
\begin{equation*}
dP_{n-1}:=\exp \{\lambda _{n-1}\int_{0}^{t_{n-1}}\langle A_{n-1}W_{s},d%
\widetilde{W_{s}^{(n)}}\rangle -\frac{\lambda _{n-1}^{2}}{2}%
\int_{0}^{t_{n-1}}|A_{n-1}W_{s}|^{2}ds\}d\widetilde{P_{n}}.
\end{equation*}%
Under $P_{n-1},$ the process
\begin{equation*}
W_{t}^{(n-1)}:=\widetilde{W_{t}^{(n)}}-\lambda
_{n-1}\int_{0}^{t}A_{n-1}W_{s}ds,\ \ \ t\in \lbrack 0,t_{n-1}]
\end{equation*}%
is a Brownian motion, and the process $\{W_{t}:t\in \lbrack 0,t_{n-1}]\}$
satisfies the following SDE:
\begin{equation*}
dW_{t}=(K_{n}(t)+\lambda _{n}A_{n}+\lambda
_{n-1}A_{n-1})W_{t}dt+dW_{t}^{(n-1)},\ \ \ t\in \lbrack 0,t_{n-1}].
\end{equation*}%
Let
\begin{eqnarray*}
C_{n-1}(t):= &&-\lambda _{n-1}^{2}A_{n-1}^{\ast }A_{n-1}-\lambda
_{n-1}[(K_{n}(t)+\lambda _{n}A_{n}^{\ast })A_{n-1} \\
&&+A_{n-1}^{\ast }(K_{n}(t)+\lambda _{n}A_{n})],\ \ \ t\in \lbrack
0,t_{n-1}],
\end{eqnarray*}%
and introduce the following matrix Riccati equation:
\begin{eqnarray*}
\frac{d}{dt}K_{n-1}(t) &=&C_{n-1}(t)-K_{n-1}(t)(K_{n}(t)+\lambda
_{n}A_{n}+\lambda _{n-1}A_{n-1}) \\
&&-(K_{n}(t)+\lambda _{n}A_{n}^{\ast }+\lambda _{n-1}A_{n-1}^{\ast
})K_{n-1}(t)-K_{n-1}^{2}(t),\ t\in \lbrack 0,t_{n-1}], \\
K_{n-1}(t_{n-1}) &=&0,
\end{eqnarray*}%
we have
\begin{align*}
& h(\lambda _{1},\cdots ,\lambda _{n})\exp \{-\frac{1}{2}%
\int_{0}^{t_{n}}Tr(K_{n}(s))ds\} \\
=& \mathbb{E}_{n-1}[\exp \{\sum_{k=1}^{n-2}\lambda
_{k}\int_{0}^{t_{k}}\langle A_{k}W_{s},dW_{s}\rangle -\frac{1}{2}%
\int_{0}^{t_{n-1}}W_{s}^{\ast }C_{n-1}(s)W_{s}ds\}] \\
=& \mathbb{E}_{n-1}[\exp \{\sum_{k=1}^{n-2}\lambda
_{k}\int_{0}^{t_{k}}\langle A_{k}W_{s},dW_{s}\rangle -\frac{1}{2}%
\int_{0}^{t_{n-1}}W_{s}^{\ast }\frac{d}{ds}K_{n-1}(s)W_{s}ds \\
& -\int_{0}^{t_{n-1}}\langle K_{n-1}(s)W_{s},(K_{n}(s)+\lambda
_{n}A_{n}+\lambda _{n-1}A_{n-1})W_{s}\rangle ds \\
& -\frac{1}{2}\int_{0}^{t_{n-1}}|K_{n-1}(s)W_{s}|^{2}ds\}] \\
=& \widetilde{E_{n-1}}[\exp \{\sum_{k=1}^{n-2}\lambda
_{k}\int_{0}^{t_{k}}\langle A_{k}W_{s},dW_{s}\rangle -\frac{1}{2}%
\int_{0}^{t_{n-1}}W_{s}^{\ast }\frac{d}{ds}K_{n-1}(s)W_{s}ds \\
& -\int_{0}^{t_{n-1}}\langle K_{n-1}(s)W_{s},(K_{n}(s)+\lambda
_{n}A_{n}+\lambda _{n-1}A_{n-1})W_{s}\rangle ds \\
& -\int_{0}^{t_{n-1}}\langle K_{n-1}(s)W_{s},dW_{s}^{(n-1)}\rangle \}] \\
=& \widetilde{E_{n-1}}[\exp \{\sum_{k=1}^{n-2}\lambda
_{k}\int_{0}^{t_{k}}\langle A_{k}W_{s},dW_{s}\rangle -\frac{1}{2}%
\int_{0}^{t_{n-1}}W_{s}^{\ast }\frac{d}{ds}K_{n-1}(s)W_{s}ds \\
& -\int_{0}^{t_{n-1}}\langle K_{n-1}(s)W_{s},dW_{s}\rangle \}].
\end{align*}%
Here we have changed the probability measure from $P_{n-1}$ to
\begin{equation*}
d\widetilde{P_{n-1}}:=\exp \{\int_{0}^{t_{n-1}}\langle
K_{n-1}(s)W_{s},dW_{s}^{(n-1)}\rangle -\frac{1}{2}%
\int_{0}^{t_{n-1}}|K_{n-1}(s)W_{s}|^{2}ds\}dP_{n-1}.
\end{equation*}%
By applying It\^{o}'s formula to the process $\{W_{t}^{\ast
}K_{n-1}(t)W_{t}:t\in \lbrack 0,t_{n-1}]\}$ and noticing that the quadratic
variation process of the semi-martingale $\{W_{t}:t\in \lbrack 0,t_{n-1}]\}$
is the same as that of a Brownian motion, we again have
\begin{equation*}
\int_{0}^{t_{n-1}}W_{s}^{\ast }\frac{d}{ds}K_{n-1}(s)W_{s}ds+2%
\int_{0}^{t_{n-1}}\langle K_{n-1}(s)W_{s},dW_{s}\rangle
+\int_{0}^{t_{n-1}}Tr(K_{n-1}(s))ds=0.
\end{equation*}%
Therefore, we arrive at
\begin{align*}
& h(\lambda _{1},\cdots ,\lambda _{n})\exp \{-\frac{1}{2}%
\int_{0}^{t_{n}}Tr(K_{n}(s))ds\} \\
=& \exp \{\frac{1}{2}\int_{0}^{t_{n-1}}Tr(K_{n-1}(s))ds\}\cdot \widetilde{%
\mathbb{E}_{n-1}}[\exp \{\sum_{k=1}^{n-2}\lambda _{k}\int_{0}^{t_{k}}\langle
A_{k}W_{s},dW_{s}\rangle \}].
\end{align*}%
Now the recursion is quite obvious from the key observation that when
applying transformation of probability measures, the original process $%
\{W_{t}\}$(over the proper time interval) is always a diffusion of the same
kind, namely, it satisfies an SDE of the form
\begin{equation*}
dW_{t}=Q(t)W_{t}dt+dB_{t},
\end{equation*}%
where $\{B_{t}\}$ is a Brownian motion under the corresponding probability
measure. To be more precise, after $j$ steps, we will have a system of $j$
matrix Riccati equations defined recursively,
\begin{equation*}
\begin{cases}
C_{n}(t) & :=-\lambda _{n}^{2}A_{n}^{\ast }A_{n},\ \ \ t\in \lbrack 0,t_{n}],
\\
\frac{d}{dt}K_{n}(t) & =C_{n}(t)-\lambda _{n}[K_{n}(t)A_{n}+A_{n}^{\ast
}K_{n}(t)]-K_{n}^{2}(t),\ \ \ t\in \lbrack 0,t_{n}], \\
K_{n}(t_{n}) & =0;%
\end{cases}%
\end{equation*}%
\begin{equation*}
\begin{cases}
C_{n-1}(t) & :=-\lambda _{n-1}^{2}A_{n-1}^{\ast }A_{n-1}-\lambda
_{n-1}[(K_{n}(t)+\lambda _{n}A_{n}^{\ast })A_{n-1}+A_{n-1}^{\ast }(K_{n}(t)
\\
& +\lambda _{n}A_{n})],\ \ \ t\in \lbrack 0,t_{n-1}], \\
\frac{d}{dt}K_{n-1}(t) & =C_{n-1}(t)-K_{n-1}(t)(K_{n}(t)+\lambda
_{n}A_{n}+\lambda _{n-1}A_{n-1}) \\
& -(K_{n}(t)+\lambda _{n}A_{n}^{\ast }+\lambda _{n-1}A_{n-1}^{\ast
})K_{n-1}(t)-K_{n-1}^{2}(t),\ \ \ t\in \lbrack 0,t_{n-1}], \\
K_{n-1}(t_{n-1}) & =0;%
\end{cases}%
\end{equation*}%
\begin{equation*}
\begin{array}{c}
\cdot \\
\cdot \\
\cdot%
\end{array}%
\end{equation*}%
\begin{equation*}
\begin{cases}
C_{n-j+1}(t) & :=-\lambda _{n-j+1}^{2}A_{n-j+1}^{\ast }A_{n-j+1}-\lambda
_{n-j+1}[(\sum_{r=n-j+2}^{n}K_{r}(t) \\
& +\sum_{r=n-j+2}^{n}(\lambda _{r}A_{r}^{\ast }))A_{n-j+1}+A_{n-j+1}^{\ast
}(\sum_{r=n-j+2}^{n}K_{r}(t) \\
& +\sum_{r=n-j+2}^{n}(\lambda _{r}A_{r}))],\ \ \ t\in \lbrack 0,t_{n-j+1}],
\\
\frac{d}{dt}K_{n-j+1}(t) & =C_{n-j+1}(t)-K_{n-j+1}(t)(%
\sum_{r=n-j+2}^{n}K_{r}(t)+\sum_{r=n-j+1}^{n}(\lambda _{r}A_{r})) \\
& -(\sum_{r=n-j+2}^{n}K_{r}(t)+\sum_{r=n-j+1}^{n}(\lambda _{r}A_{r}^{\ast
}))K_{n-j+1}(t)-K_{n-j+1}^{2}(t), \\
& t\in \lbrack 0,t_{n-j+1}], \\
K_{n-j+1}(t_{n-j+1}) & =0,%
\end{cases}%
\end{equation*}%
and we will arrive at
\begin{align*}
& h(\lambda _{1},\cdots ,\lambda _{n}) \\
=& \exp \{\frac{1}{2}\sum_{r=n-j+1}^{n}\int_{0}^{t_{r}}Tr(K_{r}(s))ds\}\cdot
\widetilde{\mathbb{E}_{n-j+1}}[\exp \{\sum_{k=1}^{n-j}\lambda
_{k}\int_{0}^{t_{k}}\langle A_{k}W_{s},dW_{s}\rangle \}].
\end{align*}%
Here under the probability measure $\widetilde{P_{n-j+1}},$ $\{W_{t}:t\in
\lbrack 0,t_{n-j+1}]\}$ satisfies the SDE
\begin{equation*}
dW_{t}=(\sum_{r=n-j+1}^{n}K_{r}(t)+\sum_{r=n-j+1}^{n}(\lambda
_{r}A_{r}))W_{t}dt+d\widetilde{W_{t}^{(n-j+1)}},\ \ \ t\in \lbrack
0,t_{n-j+1}],
\end{equation*}%
where $\{\widetilde{W_{t}^{(n-j+1)}}:t\in \lbrack 0,t_{n-j+1}]\}$ is a
Brownian motion under $\widetilde{P_{n-j+1}}.$ By carrying out a similar
argument, that is, by changing measures and applying It\^{o}'s formula, we
will obtain
\begin{align*}
& h(\lambda _{1},\cdots ,\lambda _{n})\exp \{-\frac{1}{2}\sum_{r=n-j+1}^{n}%
\int_{0}^{t_{r}}Tr(K_{r}(s))ds\} \\
=& \exp \{\frac{1}{2}\int_{0}^{t_{n-j}}Tr(K_{n-j}(s))ds\}\widetilde{\mathbb{E%
}_{n-j}}[\exp \{\sum_{k=1}^{n-j-1}\lambda _{k}\int_{0}^{t_{k}}\langle
A_{k}W_{s},dW_{s}\rangle \}],
\end{align*}%
where $\{K_{n-j}(t):t\in \lbrack 0,t_{n-j}]\}$ is the solution of the matrix
Riccati equation
\begin{eqnarray*}
\frac{d}{dt}K_{n-j}(t)
&=&C_{n-j}(t)-K_{n-j}(t)(\sum_{r=n-j+1}^{n}K_{r}(t)+\sum_{r=n-j}^{n}(\lambda
_{r}A_{r})) \\
&&-(\sum_{r=n-j+1}^{n}K_{r}(t)+\sum_{r=n-j}^{n}(\lambda _{r}A_{r}^{\ast
}))K_{n-j}(t)-K_{n-j}^{2}(t), \\
&&t\in \lbrack 0,t_{n-j}], \\
K_{n-j}(t_{n-j}) &=&0,
\end{eqnarray*}%
in which
\begin{eqnarray*}
C_{n-j}(t):= &&-\lambda _{n-j}^{2}A_{n-j}^{\ast }A_{n-j}-\lambda
_{n-j}[(\sum_{r=n-j+1}^{n}K_{r}(t)+\sum_{r=n-j+1}^{n}(\lambda
_{r}A_{r}^{\ast }))A_{n-j} \\
&&+A_{n-j}^{\ast }(\sum_{r=n-j+1}^{n}K_{r}(t)+\sum_{r=n-j+1}^{n}(\lambda
_{r}A_{r}))],\ \ \ t\in \lbrack 0,t_{n-j}].
\end{eqnarray*}%
By induction on $j$, the proof of the case where $\gamma _{1}=\cdots =\gamma
_{n}=0$ is now complete.

(2) Step two.

Now we consider the case with $\gamma_{1},\cdots,\gamma_{n}\in\mathbb{R}$.
In step one, the ultimate goal of applying those transformations of
probability measures is to eliminate the stochastic integrals one by one,
starting from the largest time interval. After each transformation, the
distribution of $W_{t_{k}}$$(k=1,\cdots,n)$ is changed. In order to work out
the case involving $\gamma_{1},\cdots,\gamma_{n},$ we need to track the
original process $\{W_{t}:t\in[0,t_{n}]\}$ after each transformation. The
main difficulty comes from the fact that if we apply a transformation on $%
[0,t_{k}],$ the distribution of $\{W_{t}\}$ over $[t_{k},t_{n}]$ is hard to
compute. However, by using the crucial observation that the diffusion form
of $\{W_{t}\}$ is invariant, w can factor out the term over $[t_{k},t_{n}].$

Let's formulate the idea in detail. By using the same notation as in step
one, we have
\begin{align*}
& g(\gamma_{1},\cdots,\gamma_{n};\lambda_{1},\lambda_{n}) \\
= & \widetilde{\mathbb{E}_{n}}[\exp\{\sum_{k=1}^{n}i\langle%
\gamma_{k},W_{t_{k}}\rangle\}\cdot\Delta_{n}],
\end{align*}
where
\begin{equation*}
\Delta_{n}:=\exp\{\frac{1}{2}\int_{0}^{t_{n}}Tr(K_{n}(s))ds+\sum_{k=1}^{n-1}%
\lambda_{k}\int_{0}^{t_{k}}\langle A_{k}W_{s},dW_{s}\rangle\}.
\end{equation*}
Under $\widetilde{P_{n}},$ the process $\{W_{t}:t\in[0,t_{n}]\}$ is a
diffusion of the form
\begin{equation*}
dW_{t}=(K_{n}(t)+\lambda_{n}A_{n})W_{t}dt+d\widetilde{W_{t}^{(n)}},\ \ \ t\in%
[0,t_{n}].
\end{equation*}
Let $\{H_{\lambda_{n}}(t):t\in[0,t_{n}]\}$ be the solution of the following
linear matrix ODE
\begin{eqnarray*}
\frac{d}{dt}H_{\lambda_{n}}(t) & = &
(K_{n}(t)+\lambda_{n}A_{n})H_{\lambda_{n}}(t),\ \ \ t\in[0,t_{n}], \\
H_{\lambda_{n}}(0) & = & Id.
\end{eqnarray*}
Then by the explicit solution of linear SDE, we have
\begin{equation*}
W_{t}=H_{\lambda_{n}}(t)\int_{0}^{t}H_{\lambda_{n}}^{-1}(s)d\widetilde{%
W_{s}^{(n)}},\ \ \ t\in[0,t_{n}].
\end{equation*}
Hence
\begin{align*}
& g(\gamma_{1},\cdots,\gamma_{n};\lambda_{1},\lambda_{n}) \\
= & \widetilde{\mathbb{E}_{n}}[\exp\{\sum_{k=1}^{n-1}i\langle%
\gamma_{k},W_{t_{k}}\rangle+i\langle\gamma_{n},H_{\lambda_{n}}(t_{n})%
\int_{0}^{t_{n}}H_{\lambda_{n}}^{-1}(s)d\widetilde{W_{s}^{(n)}}%
\rangle\}\cdot\Delta_{n}] \\
= & \widetilde{\mathbb{E}_{n}}[\exp\{\sum_{k=1}^{n-1}i\langle%
\gamma_{k},W_{t_{k}}\rangle+i\langle\gamma_{n},H_{\lambda_{n}}(t_{n})%
\int_{0}^{t_{n-1}}H_{\lambda_{n}}^{-1}(s)d\widetilde{W_{s}^{(n)}}\rangle \\
&
+i\langle\gamma_{n},H_{\lambda_{n}}(t_{n})\int_{t_{n-1}}^{t_{n}}H_{%
\lambda_{n}}^{-1}(s)d\widetilde{W_{s}^{(n)}}\rangle\}\cdot\Delta_{n}].
\end{align*}
Notice that the stochastic integral $\int_{t_{n-1}}^{t_{n}}H_{%
\lambda_{n}}^{-1}(s)d\widetilde{W_{s}^{(n)}}$ is independent of the rest
since the integrand is deterministic, we have
\begin{align*}
& g(\gamma_{1},\cdots,\gamma_{n};\lambda_{1},\lambda_{n}) \\
= & \widetilde{\mathbb{E}_{n}}[\exp\{i\langle\gamma_{n},H_{%
\lambda_{n}}(t_{n})\int_{t_{n-1}}^{t_{n}}H_{\lambda_{n}}^{-1}(s)d\widetilde{%
W_{s}^{(n)}}\rangle\}]\cdot \\
& \widetilde{\mathbb{E}_{n}}[\exp\{\sum_{k=1}^{n-1}i\langle%
\gamma_{k},W_{t_{k}}\rangle+i\langle
H_{\lambda_{n}}^{*-1}(t_{n-1})H_{\lambda_{n}}^{*}(t_{n})%
\gamma_{n},W_{t_{n-1}}\rangle\}\cdot\Delta_{n}] \\
= & R_{n}\cdot\widetilde{\mathbb{E}_{n}}[\exp\{\sum_{k=1}^{n-2}i\langle%
\gamma_{k},W_{t_{k}}\rangle+\langle\mu_{n-1},W_{t_{n-1}}\rangle\}\cdot%
\Delta_{n}],
\end{align*}
where
\begin{eqnarray*}
R_{n} & := & \widetilde{\mathbb{E}_{n}}[\exp\{i\langle\gamma_{n},H_{%
\lambda_{n}}(t_{n})\int_{t_{n-1}}^{t_{n}}H_{\lambda_{n}}^{-1}(s)d\widetilde{%
W_{s}^{(n)}}\rangle\}], \\
\mu_{n-1} & := &
\gamma_{n-1}+H_{\lambda_{n}}^{*-1}(t_{n-1})H_{\lambda_{n}}^{*}(t_{n})%
\gamma_{n}.
\end{eqnarray*}
Now we can see that the random term over $[t_{n-1}t_{n}]$ is factored out.

Similarly, by applying transformations as in step one, we further have
\begin{align*}
& g(\gamma_{1},\cdots,\gamma_{n};\lambda_{1},\cdots,\lambda_{n}) \\
= & R_{n}\cdot\widetilde{\mathbb{E}_{n-1}}[\exp\{\sum_{k=1}^{n-2}i\langle%
\gamma_{k},W_{t_{k}}\rangle+\langle\mu_{n-1},W_{t_{n-1}}\rangle\}\cdot%
\Delta_{n-1}],
\end{align*}
where
\begin{eqnarray*}
\Delta_{n-1}: & = & \exp\{\frac{1}{2}\int_{0}^{t_{n-1}}Tr(K_{n-1}(s))ds+%
\frac{1}{2}\int_{0}^{t_{n}}Tr(K_{n}(s))ds \\
& & +\sum_{k=1}^{n-2}\lambda_{k}\int_{0}^{t_{k}}\langle
A_{k}W_{s},dW_{s}\rangle\}.
\end{eqnarray*}
By a similar argument, let $\{H_{\lambda_{n-1},\lambda_{n}}(t):t\in[0,t_{n-1}%
]\}$ be the solution of the equation
\begin{eqnarray*}
\frac{d}{dt}H_{\lambda_{n-1},\lambda_{n}}(t) & = &
(K_{n}(t)+K_{n-1}(t)+\lambda_{n}A_{n}+\lambda_{n-1}A_{n-1})H_{\lambda_{n-1},%
\lambda_{n}}(t), \\
& & t\in[0,t_{n-1}], \\
H_{\lambda_{n-1},\lambda_{n}}(0) & = & Id,
\end{eqnarray*}
we obtain that
\begin{align*}
& g(\gamma_{1},\cdots,\gamma_{n};\lambda_{1},\cdots,\lambda_{n}) \\
= & R_{n}\cdot R_{n-1}\cdot\widetilde{\mathbb{E}_{n-1}}[\exp\{%
\sum_{k=1}^{n-3}i\langle\gamma_{k},W_{t_{k}}\rangle+i\langle%
\mu_{n-2},W_{t_{n-2}}\rangle\}\cdot\Delta_{n-1}],
\end{align*}
where
\begin{eqnarray*}
R_{n-1} & := & \widetilde{\mathbb{E}_{n-1}}[\exp\{i\langle\mu_{n-1},H_{%
\lambda_{n-1},\lambda_{n}}(t_{n-1})\int_{t_{n-2}}^{t_{n-1}}H_{\lambda_{n-1},%
\lambda_{n}}^{-1}(s)d\widetilde{W_{s}^{(n-1)}}\rangle\}], \\
\mu_{n-2} & := &
\gamma_{n-2}+H_{\lambda_{n-1},\lambda_{n}}^{*-1}(t_{n-2})H_{\lambda_{n-1},%
\lambda_{n}}^{*}(t_{n-1})\mu_{n-1}.
\end{eqnarray*}

Finally, by a simple induction argument, we arrive at
\begin{align*}
g(\gamma_{1},\cdots,\gamma_{n};\lambda_{1},\cdots,\lambda_{n}) &
=\prod_{j=1}^{n}(R_{j}\cdot\exp\{\frac{1}{2}\int_{0}^{t_{j}}Tr(K_{j}(s))ds%
\}).
\end{align*}
Here
\begin{eqnarray*}
R_{j}: & = & \widetilde{\mathbb{E}_{j}}[\exp\{i\langle\mu_{j},H_{%
\lambda_{j,}\cdots,\lambda_{n}}(t_{j})\int_{t_{j-1}}^{t_{j}}H_{\lambda_{j},%
\cdots,\lambda_{n}}^{-1}(s)d\widetilde{W_{s}^{(j)}}\}],\ \ \ j=1,2,\cdots,n,
\\
\end{eqnarray*}
$\{H_{\lambda_{j,}\cdots,\lambda_{n}}(t):t\in[0,t_{j}]\}$ is the solution of
the equation
\begin{eqnarray*}
\frac{d}{dt}H_{\lambda_{j,}\cdots,\lambda_{n}}(t) & = &
(\sum_{r=j}^{n}K_{j}(t)+\sum_{r=j}^{n}(\lambda_{j}A_{j}))H_{\lambda_{j,}%
\cdots,\lambda_{n}}(t),\ \ \ t\in[0,t_{j}], \\
H_{\lambda_{j,}\cdots,\lambda_{n}}(0) & = & Id,
\end{eqnarray*}
and $\{\mu_{j}:j=1,2,\cdots,n\}$ is defined recursively by
\begin{eqnarray*}
\mu_{n} & := & \gamma_{n}, \\
\mu_{j} & := &
\gamma_{j}+H_{\lambda_{j+1},\cdots,\lambda_{n}}^{*-1}(t_{j})H_{%
\lambda_{j+1},\cdots,\lambda_{n}}^{*}(t_{j+1})\mu_{j+1},\ \ \
j=1,2,\cdots,n-1.
\end{eqnarray*}
It remains to compute $R_{j}$$(j=1,2,\cdots,n)$ explicitly, which is easy
since everything here is Gaussian. Namely, we have
\begin{eqnarray*}
R_{j} & = & \widetilde{\mathbb{E}_{j}}[\exp\{i\langle
H_{\lambda_{j},\cdots,\lambda_{n}}^{*}(t_{j})\mu_{j},%
\int_{t_{j-1}}^{t_{j}}H_{\lambda_{j},\cdots,\lambda_{n}}^{-1}(s)d\widetilde{%
W_{s}^{(n)}}\}] \\
& = & \widetilde{\mathbb{E}_{j}}[\exp\{i\int_{t_{j-1}}^{t_{j}}\langle
H_{\lambda_{j},\cdots,\lambda_{n}}^{*-1}(s)H_{\lambda_{j},\cdots,%
\lambda_{n}}^{*}(t_{j})\mu_{j},d\widetilde{W_{s}^{(n)}}\rangle\}] \\
& = & \exp\{-\frac{1}{2}\int_{t_{j-1}}^{t_{j}}|H_{\lambda_{j},\cdots,%
\lambda_{n}}^{*-1}(s)H_{\lambda_{j},\cdots,\lambda_{n}}^{*}(t_{j})%
\mu_{j}|^{2}ds\}.
\end{eqnarray*}
Therefore, the proof is now complete and we have
\begin{align*}
& g(\gamma_{1},\cdots,\gamma_{n};\lambda_{1},\cdots,\lambda_{n}) \\
= & \prod_{j=1}^{n}\exp\{\frac{1}{2}\int_{0}^{t_{j}}Tr(K_{j}(s))ds-\frac{1}{2%
}\int_{t_{j-1}}^{t_{j}}|H_{\lambda_{j},\cdots,\lambda_{n}}^{*-1}(s)H_{%
\lambda_{j},\cdots,\lambda_{n}}^{*}(t_{j})\mu_{j}|^{2}ds\}.
\end{align*}
\end{proof}

From the proof of the above proposition, we can see that the computation of
\begin{equation*}
\mathbb{E}[\exp\{\sum_{k=1}^{n}i\langle\gamma_{k},W_{t_{k}}\rangle+%
\sum_{k=1}^{n}\lambda_{k}L_{t_{k}}^{A_{k}}\}]
\end{equation*}
for small $\lambda_{1},\cdots,\lambda_{n}\in\mathbb{R}$ reduces to the
solution of a recursive system of symmetric matrix Riccati equations and the
solution of a system of independent first order linear matrix ODEs. If $%
\gamma_{1}=\cdots=\gamma_{n}=0,$ then we don't need the ODE system at all.
Now we are going to complexify our case before.

\begin{lemma}
Fix $\gamma_{1},\cdots,\gamma_{n}\in\mathbb{R}.$ Then when $c$ is small
enough, the function
\begin{equation*}
\phi(z_{1},\cdots,z_{n}):=\mathbb{E}[\exp\{\sum_{k=1}^{n}i\langle%
\gamma_{k},W_{t_{k}}\rangle+\sum_{k=1}^{n}z_{k}L_{t_{k}}^{A_{k}}\}]
\end{equation*}
is holomorphic in the domain $D_{c}:=\{(z_{1},\cdots,z_{n})\in\mathbb{C}%
^{n}:Re(z_{j})\in(-c,c),j=1,2,\cdots,n\}$ of $\mathbb{C}^{n}$. Moreover, the
function
\begin{eqnarray*}
\psi(\lambda_{1},\cdots,\lambda_{n}): & = & \prod_{j=1}^{n}\exp\{\frac{1}{2}%
\int_{0}^{t_{j}}Tr(K_{j}(s))ds \\
& & -\frac{1}{2}\int_{t_{j-1}}^{t_{j}}|H_{\lambda_{j},\cdots,%
\lambda_{n}}^{*-1}(s)H_{\lambda_{j},\cdots,\lambda_{n}}^{*}(t_{j})%
\mu_{j}|^{2}ds\}
\end{eqnarray*}
defined on $\mathbb{R}^{n}$ can be extended holomorphically to$\mathbb{C}%
^{n} $. Such an extension is unique, and when restricted to $D_{c}$,
\begin{equation*}
\phi(z_{1},\cdots,z_{n})=\psi(z_{1},\cdots,z_{n}).
\end{equation*}
\end{lemma}

\begin{proof}
By lemma 1, when $c$ is small enough, $\phi(z_{1},\cdots,z_{n})$ is well
defined on $D_{c}.$ The continuity of $\phi(z_{1},\cdots,z_{n})$ follows
easily from uniform integrability. Moreover, since the function
\begin{equation*}
(z_{1},\cdots,z_{n})\mapsto\exp\{\sum_{k=1}^{n}i\langle\gamma_{k},W_{t_{k}}%
\rangle+\sum_{k=1}^{n}z_{k}L_{t_{k}}^{A_{k}}\}
\end{equation*}
is holomorphic on $\mathbb{C}^{n}$ for every $\omega\in\Omega,$ by Fubini's
theorem and Morera's theorem, $\phi(z_{1},\cdots,z_{n})$ is holomorphic in $%
D_{c}.$

On the other hand, it is obvious that the recursive system of matrix Riccati
equations and the system of independent matrix ODEs defined in proposition 2
depend analytically on $\lambda_{1},\cdots,\lambda_{n}\in\mathbb{R}$ and
extend naturally to the case where $\lambda_{1},\cdots,\lambda_{n}\in\mathbb{%
C}.$ Consequently, when $(\lambda_{1},\cdots,\lambda_{n})$ is replaced by $%
(z_{1},\cdots,z_{n})\in\mathbb{C}^{n}$, the two systems determine solutions
depending holomorphically on $z_{1},\cdots,z_{n}.$ It follows that $%
\psi(\lambda_{1},\cdots,\lambda_{n})$ possesses a unique holomorphic
extension to $\mathbb{C}^{n}$.

Finally, since $\phi$and $\psi$ coincide in the set $\{(\lambda_{1},\cdots,%
\lambda_{n})\in\mathbb{R}^{n}:\lambda_{j}\in(-c,c),j=1,2,\cdots,n\},$ by the
identity theorem, they coincide in $D_{c}$.
\end{proof}

With the preparations above, the proof of our main result on the formula for
the joint characteristic function $f(\gamma_{1},\cdots,\gamma_{n};%
\Lambda_{1},\cdots,\Lambda_{n})$ defined in section one now follows easily.
In fact, the result follows immediately from proposition 2 and lemma 3 with
the observation that the set
\begin{equation*}
\{(i\Lambda_{1},\cdots,i\Lambda_{n}):\Lambda_{j}\in\mathbb{R}%
,j=1,2,\cdots,n\}
\end{equation*}
is contained in $D_{c}$.

\section{An Example: the Two Dimensional L\'{e}vy's Stochastic Area Process}

In this section, we are going to apply our result to study the two
dimensional L\'{e}vy's stochastic area process, first introduced by L\'{e}vy
in {\cite{MR0002734}}. Namely, we consider the case where $d=2$ and
\begin{equation*}
A=\left(
\begin{array}{cc}
0 & -1 \\
1 & 0%
\end{array}%
\right) .
\end{equation*}%
The L\'{e}vy's stochastic area process is given by
\begin{equation*}
L_{t}=\int_{0}^{t}(W_{s}^{(1)}dW_{s}^{(2)}-W_{s}^{(2)}dW_{s}^{(1)}),\ \ \
t\geqslant 0.
\end{equation*}%
We try to derive an explicit formula for the finite dimensional joint
characteristic function of the coupled process $\{(W_{t},L_{t}):t\geqslant
0\}$.

Throughout this section, $0=t_{0}<t_{1}<\cdots<t_{n}$ will be fixed.

In section 2, the computation of the finite dimensional joint characteristic
function of the coupled process $\{(W_{t},L_{t}^{A}):t\geqslant0\}$ reduces
to the solution of a recursive system of symmetric Riccati equations and a
system of independent first order linear matrix ODEs. In the case here, we
will see that the Riccati system is actually real and scalar, and the linear
system is explicitly solvable. In fact, we have:

\begin{proposition}
For the two dimensional L\'{e}vy's stochastic area process $%
\{L_{t}:t\geqslant 0\}$, by using the same notation as in theorem 1 with the
assumption
\begin{equation*}
A_{1}=\cdots =A_{n}=A,
\end{equation*}%
the solution matrices of the Riccati system are real diagonal matrices with
identical diagonal entries, and the Riccati system essentially reduces to a
system of real scalar Riccati equations recursively defined from $j=n$ to $%
j=1$ by
\begin{eqnarray*}
\frac{d}{dt}k_{i\Lambda _{j},\cdots ,i\Lambda _{n}}(t) &=&(\Lambda
_{j}^{2}+2\Lambda _{j}\sum_{r=j+1}^{n}\Lambda
_{r})-2(\sum_{r=j+1}^{n}k_{i\Lambda _{r},\cdots ,i\Lambda
_{n}}(t))k_{i\Lambda _{j},\cdots ,i\Lambda _{n}}(t) \\
&&-k_{i\Lambda _{j},\cdots ,i\Lambda _{n}}^{2}(t),\ \ \ t\in \lbrack
0,t_{j}], \\
k_{i\Lambda _{j},\cdots ,i\Lambda _{n}}(t_{j}) &=&0.
\end{eqnarray*}%
Moreover, the linear system in theorem 1 is explicitly solvable, namely, for
$j=1,2,\cdots ,n,$
\begin{equation*}
H_{i\Lambda _{j},\cdots ,i\Lambda _{n}}(t)=\exp
\{\int_{0}^{t}(\sum_{r=j}^{n}K_{i\Lambda _{r},\cdots ,i\Lambda
_{n}}(s)+i(\sum_{r=j}^{n}\Lambda _{r})A)ds\},\ \ \ t\in \lbrack 0,t_{j}].
\end{equation*}
\end{proposition}

\begin{proof}
We first consider the Riccati system. For $j=n$,
\begin{equation*}
C_{i\Lambda_{n}}(t)=\left(%
\begin{array}{cc}
\Lambda_{n}^{2} & 0 \\
0 & \Lambda_{n}^{2}%
\end{array}%
\right),\ \ \ t\in[0,t_{n}],
\end{equation*}
and the Riccati equation is defined by
\begin{eqnarray*}
\frac{d}{dt}K_{i\Lambda_{n}}(t) & = & \left(%
\begin{array}{cc}
\Lambda_{n}^{2} & 0 \\
0 & \Lambda_{n}^{2}%
\end{array}%
\right)-i\Lambda_{n}\cdot(K_{i\Lambda_{n}}(t)\left(%
\begin{array}{cc}
0 & -1 \\
1 & 0%
\end{array}%
\right) \\
& & +\left(%
\begin{array}{cc}
0 & 1 \\
-1 & 0%
\end{array}%
\right)K_{i\Lambda_{n}}(t))-K_{i\Lambda_{n}}^{2}(t),\ \ \ t\in[0,t_{n}], \\
K_{i\Lambda_{n}}(t_{n}) & = & 0.
\end{eqnarray*}
By uniqueness, it is easy to see that the solution of the above equation is
given by a matrix of the form
\begin{equation*}
K_{i\Lambda_{n}}(t)=\left(%
\begin{array}{cc}
k_{i\Lambda_{n}}(t) & 0 \\
0 & k_{i\Lambda_{n}}(t)%
\end{array}%
\right),\ \ \ t\in[0,t_{n}],
\end{equation*}
where $\{k_{i\Lambda_{n}}(t):t\in[0,t_{n}]\}$ solves the real scalar Riccati
equation
\begin{eqnarray*}
\frac{d}{dt}k_{i\Lambda_{n}}(t) & = &
\Lambda_{n}^{2}-k_{i\Lambda_{n}}^{2}(t),\ \ \ t\in[0,t_{n}], \\
k_{i\Lambda_{n}}(t_{n}) & = & 0.
\end{eqnarray*}
For $j=n-1,$ by easy computation we have
\begin{equation*}
C_{i\Lambda_{n-1},i\Lambda_{n}}(t)=\left(%
\begin{array}{cc}
\Lambda_{n}^{2}+2\Lambda_{n-1}\Lambda_{n} & 0 \\
0 & \Lambda_{n}^{2}+2\Lambda_{n-1}\Lambda_{n}%
\end{array}%
\right),\ \ \ t\in[0,t_{n-1}].
\end{equation*}
It follows from uniqueness and the case $j=n$ that $K_{i\Lambda_{n-1},i%
\Lambda_{n}}(t)$ is also a diagonal matrix of the form
\begin{equation*}
K_{i\Lambda_{n-1},i\Lambda_{n}}(t)=\left(%
\begin{array}{cc}
k_{i\Lambda_{n-1},i\Lambda_{n}}(t) & 0 \\
0 & k_{i\Lambda_{n-1},i\Lambda_{n}}(t)%
\end{array}%
\right),\ \ \ t\in[0,t_{n-1}],
\end{equation*}
where $\{k_{i\Lambda_{n-1},i\Lambda_{n}}(t):t\in[0,t_{n-1}]\}$ solves the
real scalar Riccati equation
\begin{eqnarray*}
\frac{d}{dt}k_{i\Lambda_{n},i\Lambda_{n-1}}(t) & = &
(\Lambda_{n}^{2}+2\Lambda_{n-1}\Lambda_{n})-2k_{i\Lambda_{n}}(t)k_{i%
\Lambda_{n-1},i\Lambda_{n}}(t) \\
& & -k_{i\Lambda_{n-1},i\Lambda_{n}}^{2}(t),\ \ \ t\in[0,t_{n-1}], \\
k_{i\Lambda_{n-1},i\Lambda_{n}}(t_{n-1}) & = & 0.
\end{eqnarray*}
The rest of the argument follows from recursion easily (the crucial
observation is that $C_{i\Lambda_{j},\cdots,i\Lambda_{n}}(t)$ is always a
real diagonal matrix with identical diagonal entries).

The second part of the proposition follows immediately from the fact that
for $j=1,2,\cdots,n,$ if we denote $\Phi_{j}(t)$ as the coefficient matrix
of the $j-$th linear ODE of the independent system, then
\begin{equation*}
\Phi_{j}(s)\Phi_{j}(t)=\Phi_{j}(t)\Phi_{j}(s),\ \ \ s,t\in[0,t_{j}].
\end{equation*}
\end{proof}

Now we are going to study the finite dimensional joint characteristic
function $$f(\gamma_{1},\cdots,\gamma_{n};\Lambda_{1},\cdots,\Lambda_{n})$$ of
$\{(W_{t},L_{t}):t\geqslant0\}$. To simplify our notation, we use $%
\{k_{j},H_{j}:j=1,\cdots,n\}$ to denote the solution $\{k_{i\Lambda_{j},%
\cdots,i\Lambda_{n}},H_{i\Lambda_{j},\cdots,i\Lambda_{n}}:j=1,\cdots,n\}$ in
proposition 5. Assume first that (nondegeneracy)
\begin{equation*}
\Lambda_{j}+\cdots+\Lambda_{n}\neq0.\ \ \ j=1,\cdots,n.
\end{equation*}

We first study the process $\{L_{t}:t\geqslant 0\}$. A crucial observation
is that for $j=1,\cdots ,n$, by adding together from the $j-$th equation to
the $n-$th equation in the scalar Riccati system in proposition 5, we obtain
a neat scalar Riccati equation without linear terms:
\begin{equation*}
\frac{d}{dt}(\sum_{r=j}^{n}k_{r}(t))=(\sum_{r=j}^{n}\Lambda
_{r})^{2}-(\sum_{r=j}^{n}k_{r}(t))^{2},\ \ \ t\in \lbrack 0,t_{j}],
\end{equation*}%
in which the unique solution is determined by the terminal data at $t=t_{j}.$
Let
\begin{equation*}
c_{j}:=\sum_{r=j}^{n}\Lambda _{r},\ s_{j}(t):=\sum_{r=j}^{n}k_{j}(t),\ \ \
t\in \lbrack 0,t_{j}],j=1,\cdots ,n,
\end{equation*}%
then it is not hard to derive that
\begin{eqnarray*}
s_{j}(t) &=&c_{j}\frac{c_{j}\sinh (c_{j}(t-t_{j}))+s_{j}(t_{j})\cosh
(c_{j}(t-t_{j}))}{c_{j}\cosh (c_{j}(t-t_{j}))+s_{j}(t_{j})\sinh
(c_{j}(t-t_{j}))}, \\
&&t\in \lbrack 0,t_{j}],j=1,\cdots ,n,
\end{eqnarray*}%
where $\{s_{j}(t_{j})\}_{j=1}^{n}$ is defined recursively by
\begin{eqnarray*}
s_{n}(t_{n}) &=&0, \\
s_{j-1}(t_{j-1}) &=&c_{j}\frac{c_{j}\sinh
(c_{j}(t_{j-1}-t_{j}))+s_{j}(t_{j})\cosh (c_{j}(t_{j-1}-t_{j}))}{c_{j}\cosh
(c_{j}(t_{j-1}-t_{j}))+s_{j}(t_{j})\sinh (c_{j}(t_{j-1}-t_{j}))}, \\
&&j=2,3,\cdots ,n.
\end{eqnarray*}%
Now apply theorem 1, we have
\begin{align*}
& \mathbb{E}[\exp \{\sum_{k=1}^{n}i\Lambda _{k}L_{t_{k}}\}] \\
=& \exp \{\sum_{j=1}^{n}\frac{1}{2}\int_{0}^{t_{j}}Tr(K_{i\Lambda
_{j},\cdots ,i\Lambda _{n}}(s))ds\} \\
=& \exp \{\sum_{j=1}^{n}\int_{0}^{t_{j}}k_{j}(s)ds\} \\
=& \exp \{\sum_{j=1}^{n}\int_{t_{j-1}}^{t_{j}}s_{j}(u)du\} \\
=& \prod_{j=1}^{n}\frac{c_{j}}{c_{j}\cosh
(c_{j}(t_{j-1}-t_{j}))+s_{j}(t_{j})\sinh (c_{j}(t_{j-1}-t_{j}))},\ \ \ \ \ \
\ \ \ \ (\ast )
\end{align*}%
which seems, to our best knowledge, not appear in literatures, though it
should follow from the Markov property and L\'{e}vy's formula.

It should be pointed out that the nondegeneracy assumption on $%
\{\Lambda_{j}:j=1,\cdots,n\}$ is not important. In fact, if for some $j,$
\begin{equation*}
c_{j}=0,
\end{equation*}
and assume that $\{s_{j+1}(t):t\in[0,t_{j+1}]\}$ has been solved, then $%
\{s_{j}(t):t\in[0,t_{j}]\}$ can be sovled as
\begin{equation*}
s_{j}(t)=\frac{s_{j}(t_{j})}{1+s_{j}(t_{j})(t-t_{j})}=\frac{s_{j+1}(t_{j})}{%
1+s_{j+1}(t_{j})(t-t_{j})},\ \ \ t\in[0,t_{j}].
\end{equation*}
It is easy to see that $\{s_{j}(t):t\in[0,t_{j}]\}$ is actually the limit of
the nondegenerate case as $c_{j}\rightarrow0.$ Moreover, the corresponding
term in the product $(*)$ can be written as
\begin{equation*}
\frac{1}{1+s_{j}(t_{j})(t_{j-1}-t_{j})},
\end{equation*}
which is also the limit of the nondegenerate case. Therefore, we still use
the same notation even in degenerate cases.

Now consider the coupled process $\{(W_{t},L_{t}):t\geqslant0\}.$ By
proposition 5 and the above computation, for $j=1,\cdots,n$, we can solve
the linear system explicitly to obtain
\begin{align*}
H_{j}(t)= & \exp\{\left(%
\begin{array}{cc}
a_{j}(t) & -ic_{j}t \\
ic_{j}t & a_{j}(t)%
\end{array}%
\right)\},\ \ \ t\in[0,t_{j}],
\end{align*}
where
\begin{eqnarray*}
a_{j}(t): & = & \int_{0}^{t}s_{j}(u)du \\
& = & \ln\frac{c_{j}\cosh(c_{j}(t-t_{j}))+s_{j}(t_{j})\sinh(c_{j}(t-t_{j}))}{%
c_{j}\cosh(c_{j}t_{j})-s_{j}(t_{j})\sinh(c_{j}t_{j})},\ \ \ t\in[0,t_{j}].
\end{eqnarray*}
Here the formula for $\{a_{j}(t):t\in[0,t_{j}]\}$ works in the degenerate
case where $c_{j}=0$ as well.

Finally, by theorem 1, we can write down the finite dimensional joint
characteristic function of the coupled process $\{(W_{t},L_{t}):t\geqslant0%
\} $. Namely, we have

\begin{theorem}
For the coupled process $\{(W_{t},L_{t}):t\geqslant 0\},$ the finite
dimensional joint characteristic function $f(\gamma _{1},\cdots ,\gamma
_{n};\Lambda _{1},\cdots ,\Lambda _{n})$ is given by
\begin{align*}
f(\gamma _{1},\cdots ,\gamma _{n};\Lambda _{1},\cdots ,\Lambda _{n})=&
\prod_{j=1}^{n}\frac{c_{j}}{c_{j}\cosh
(c_{j}(t_{j-1}-t_{j}))+s_{j}(t_{j})\sinh (c_{j}(t_{j-1}-t_{j}))}\cdot \\
& \exp \{-\frac{1}{2}\int_{t_{j-1}}^{t_{j}}\langle H_{j}^{\ast
-1}(s)H_{j}^{\ast }(t_{j})\mu _{j}\rangle {}^{2}ds\},
\end{align*}%
where $\{c_{j},s_{j}(t_{j}),H_{j}:j=1,\cdots ,n\}$ is defined previously in
this section, and $\{\mu _{j}:j=1,\cdots ,n\}$ is defined recursively in
terms of $\{\gamma _{j},H_{j}:j=1,\cdots ,n\}$ as in theorem 1.
\end{theorem}


\providecommand{\bysame}{\leavevmode\hbox to3em{\hrulefill}\thinspace} %
\providecommand{\MR}{\relax\ifhmode\unskip\space\fi MR }
\providecommand{\MRhref}[2]{  \href{http://www.ams.org/mathscinet-getitem?mr=#1}{#2}
} \providecommand{\href}[2]{#2}

\end{document}